\documentclass[journal]{IEEEtran}      

\usepackage{graphicx}
\usepackage{algorithm}
\usepackage{algorithmic}
\usepackage{amsfonts,amssymb}
\usepackage{bm}
\usepackage{amsmath}
\usepackage{amsthm}
\usepackage{empheq}
\usepackage{cases}
\usepackage{arydshln}
\usepackage{dsfont}
\usepackage[colorlinks,
linkcolor=red,
anchorcolor=blue,
citecolor=blue
]{hyperref}
\usepackage{cite}
\usepackage{subcaption}

\newtheorem{thm}{Theorem}
\newtheorem{exm}{Example}
\newtheorem{lem}{Lemma}
\newtheorem{rem}{Remark}
\newtheorem{pro}{Proposition}

\newtheorem{defn}{Definition}

\newtheorem{prob}{Problem}

\newcommand{\diag}{\mathrm{Diag}}

\newcommand{\tr}{\mathrm{tr}}

\begin{document}

\title{\LARGE \bf
	Reduced Order Modeling of Diffusively Coupled Network Systems: \\An Optimal Edge Weighting Approach
}

\author{
	Xiaodong~Cheng, 
	Lanlin Yu, Dingchao Ren
	and Jacquelien M.A. Scherpen  
	\thanks{This work of X. Cheng is supported by 
		the European Research Council (ERC), Advanced Research Grant SYSDYNET, under the European Unions Horizon 2020 research and innovation programme (Grant Agreement No. 694504). This work of L. Yu is supported by the National Natural Science Foundation of China Under Project 61761136005.}%
	
	\thanks{Xiaodong Cheng is with Control Systems Group, Department of Electrical Engineering,
		Eindhoven University of Technology, 
		5600 MB Eindhoven,
		The Netherlands
		{\tt\small x.cheng@tue.nl}}
	\thanks{Lanlin Yu is with the Institute of Advanced Technology, Westlake Institute for Advanced Study, Westlake University, Hangzhou 310024, China and Dingchao Ren is with the Department of Automation, University of Science and Technology of China, Hefei 230026, China.
	{\tt\small \{yulanlin1992, dcrenustc\}@gmail.com}}
	\thanks{Jacquelien M.A. Scherpen is with Jan C. Willems Center for Systems and Control, Engineering and Technology Institute Groningen, Faculty of Science and Engineering, University of Groningen, Nijenborgh 4, 9747 AG Groningen, the Netherlands.
		{\tt\small j.m.a.scherpen@rug.nl}}

}

\maketitle

\begin{abstract}
This paper studies reduced-order modeling of dynamic networks with strongly connected topology. Given a graph clustering of an original complex network, we construct a quotient graph with less number of vertices, where the edge weights are parameters to be determined. The model of the reduced network is thereby obtained with parameterized system matrices, and then an edge weighting procedure is devised, aiming to select an optimal set of edge weights that minimizes the approximation error between the original and the reduced-order network models in terms of $\mathcal{H}_2$-norm. The effectiveness of the proposed method is illustrated by a numerical example.
\end{abstract}

\section{Introduction} 

With the growing complexity of spatially interconnected dynamic systems, the importance of understanding and managing dynamic networks has been widely recognized. An important class of dynamic networks is given by the so-called diffusively coupled networks, which are commonly used for describing diffusion processes, e.g., information or energy spreading in networks. The examples can be found in vehicle formations, electrical networks, synchronization in sensor networks, and opinion dynamics in social networks, see e.g. \cite{fax2001graph,dorfler2018electrical,Scutari2008SensorNetworks,proskurnikov2015opinion}. The spatial structures of such systems are usually complex and result in high-dimensional models that cause challenges for analysis, control, and optimization. To effectively capture the collective behaviors of dynamics over networks, it is desirable to simplify the structure of a complex network without a significant loss of accuracy. 

Different from model reduction problems for other types of dynamic systems, the one considered in this paper puts emphasis on the preservation of the network structure, which 
is necessary for applications e.g., distributed controller design and sensor allocation \cite{summers2014optimal,gates2016control,ishizaki2018graph}. Conventional model reduction methods, e.g., balanced truncation and moment matching, merely focus on approximating the input-output behavior of a given dynamic system \cite{antoulas2005approximation}, while the preservation of the network structure is barely guaranteed. Although a generalized balanced truncation approach in \cite{XiaodongBT2017} is able to construct an accurate reduced-order model with a network interpretation, the relation between the original and obtained new typologies is not yet clear. 
Singular perturbation approximation, as alternative network structure-preserving approach, has been applied to the reduction of electric circuits \cite{Biyik2008Areaaggregation} and chemical reaction networks \cite{Rao2013graph}. This class of methods mainly relies on time-scale separation of the states in an autonomous network system, while the external inputs are not considered explicitly. Besides, the resulting reduced topology hardly retains sparsity. 

Recently, clustering-based methods have been intensively studied and become the mainstream methodology for reducing network systems, see e.g., \cite{Monshizadeh2014,Besselink,Ishizaki2014,ishizaki2015clustereddirected,jongsma2018model,XiaodongTAC2018MAS,XiaodongCDC2017Digraph,Xiaodong2020Digraph,Martin2018scalefree}. With graph clustering, the vertices in a large-scale network are partitioned into several disjoint clusters. This class of methods has a clear advantage in retaining the consensus property \cite{Monshizadeh2014,Xiaodong2020Digraph}, system positivity  \cite{ishizaki2015clustereddirected}, and scale-free property \cite{Martin2018scalefree} in reduced-order models. The model reduction procedure can be implemented via the Petrov-Galerkin projection framework, where the projection matrix is formed based on the vertex clusters. However, all the current clustering-based methods put their main focus on finding suitable clusters. After clusters are found, reduced-order network models are then directly determined by the projection framework, while the freedom to construct a reduced-order network model with higher accuracy is overlooked.

In this paper, we will explore the latter freedom and provide a novel method for reduced-order modeling of directed networks. We do not aim to find an optimal clustering. Instead, we assume that the clustering of a network is given, which leads to a \textit{quotient graph}. A parameterized reduced-order model is established based on this quotient graph, in which the edge weights are free variables to be optimized. Then, the major problem in this paper follows: How to tune the edge weights in the parameterized reduced-order  model to minimize the approximation error?

This problem can be formulated as an optimization problem with the objective to minimize the $\mathcal{H}_{2}$-norm of the reduction error between original and reduced network systems, in which the edge weights of the reduced network are variables to be optimized. This edge weighting problem is subject to a bilinear matrix inequality (BMI) constraint, which is computationally expensive. Therefore, we devise a novel edge weighting algorithm based on the \textit{convex-concave decomposition}, which linearizes the nonconvex constraint as a convex one in the form of a linear matrix
inequality (LMI). An iterative scheme is implemented to search for a set of optimal weights. The convergence of this algorithm is theoretically ensured, and thus at least a local optimum can be reached. Moreover, we initialize the edge weights as the outcome of clustering-based projection, such that the obtained reduced-order network model is guaranteed a better approximation accuracy than the clustering-based projection methods. 
 
The rest of this paper is organized as follows. In Section \ref{sec:Preliminaries}, we recap some preliminaries on graph theory and introduce the problem setup. In Section~\ref{sec:Main}, the parameterized reduced-order model is formulated, and an edge weighting algorithm is proposed to minimize the approximation error. In Section~\ref{sec:example}, the proposed method is illustrated by an example, and Section~\ref{sec:conclusion} finally makes some concluding remarks.

\textit{Notation:} The symbol $\mathbb{R}$ and $\mathbb{R}_+$ denote the set of real numbers and positive real numbers, respectively. Let $\mathcal{S}^n$ be the set of real symmetric matrices of size $n \times n$.  $I_n$ is the identity matrix of size $n$, and $\mathds{1}_n$ represents the vector in $\mathbb{R}^n$ of all ones. The cardinality of a set $\mathcal{S}$ is denoted by $|\mathcal{S}|$. For a real matrix $A$, the columns of $A^{\perp}$ form a basis of the null space of
$A$, that is, $AA^{\perp}=0$. 
\section{Preliminaries and Problem Setting} 
\label{sec:Preliminaries}
This section provides necessary definitions and concepts in graph theory used in this paper, and we refer to \cite{mesbahi2010graph} for more details. The model of a dynamical network is then introduced and the model reduction problem is formulated. 

\subsection{Graph Theory}
\label{sec:graphclustering}
A directed graph $\mathcal{G}: = (\mathcal{V}, \mathcal{E})$ consists of a finite and nonempty node set  $\mathcal{V}: = \{1, 2, \cdots, n\}$ and an edge set $\mathcal{E} \subseteq \mathcal{V} \times \mathcal{V}$. Each element  in $\mathcal{E}$ is an ordered pair of $\mathcal{V}$, and if $(i,j) \in \mathcal{E}$, we say that the edge is directed from vertex $i$ to vertex $j$. A directed graph $\mathcal{G}$ is called \textit{simple}, if $\mathcal{G}$ does not contain self-loops (i.e., $\mathcal{E}$ does not contain any edge of the form $(i,i)$, $\forall~i\in \mathcal{V}$), and there exists only one edge directed from $i$ to $j$, if $(i,j) \in \mathcal{E}$. 

Next, we introduce several important matrices for characterizing a directed simple graph. Let $m: = |\mathcal{E}|$, the \textit{incidence matrix} $\mathcal{B} \in \mathbb{R}^{n \times m}$  is defined by
\begin{equation*}
	\mathcal{B}_{ij} = \begin{cases} 
	+1 & \text{if edge}~j~\text{is directed from vertex}~i, \\ 
	-1 & \text{if edge}~j~\text{is directed to vertex}~i, \\ 
	0 & \text{otherwise}. \end{cases} 
\end{equation*}
If each edge is assigned a positive value (weight), then the \textit{weighted adjacency matrix} of $\mathcal{G}$, denoted by $\mathcal{A}$, is defined such that $\mathcal A_{ij} \in \mathbb{R}_+$ denotes the weight of edge $(j,i) \in \mathcal{E}$, and $\mathcal A_{ij} = 0$ if $(j,i) \notin \mathcal{E}$. In the case of a simple graph, $\mathcal{A}$ is a binary matrix with zeros on its diagonal.
Then, the \textit{Laplacian matrix} ${L} \in \mathbb{R}^{n \times n}$ of the graph $\mathcal{G}$ is defined as
\begin{equation} \label{defn:Laplacian}
{L}_{ij} = \left\{ \begin{array}{ll} 
\sum_{j=1,j\ne i}^{n} \mathcal A_{ij} & \text{if}~i = j,\\
-\mathcal A_{ij}  & \text{otherwise.}
\end{array}
\right.
\end{equation}
Clearly, ${L} \mathds{1} = 0$. The diagonal entries of ${L}$  are strictly positive, and the off-diagonal entries are non-positive.
Alternatively, we can characterize the Laplacian matrix using the incidence matrix of $\mathcal{G}$ as
\begin{equation} \label{eq:LEWB}
	{L} = \mathcal{B}_0 W \mathcal{B}^\top,
\end{equation}
where $\mathcal{B}_0$ is a binary matrix obtained by replacing all ``$-1$" entries in the incidence matrix $\mathcal{B}$ with zeros, and 
\begin{equation*}
	W := \diag(w), \ \text{with}\ w = 
	\begin{bmatrix}
		w_1 & w_2 & \cdots & w_{|\mathcal{E}|}
	\end{bmatrix}^\top,
\end{equation*}
and $w_k$ the positive weight associated to the $k$-th edge, for all $k \in \{1,2,\cdots, |\mathcal{E}|\} $. 

For a vertex in a weighted graph, the \textit{indegree} and \textit{outdegree} of the vertex are computed as $\sum_{j \in \mathcal{V}} \mathcal{A}_{ij}$ and $\sum_{i \in \mathcal{V}} \mathcal{A}_{ij}$, respectively. A strongly connected graph $\mathcal{G}$ is called \textit{balanced} if the indegree and outdegree of each vertex in $\mathcal{G}$ is equal. From \eqref{defn:Laplacian}, the following lemma is immediate.
\begin{lem} \label{lem:balanced}
	A weighted strongly connected graph $\mathcal{G}$ is balanced if and only if one of the following conditions hold. 
	\begin{enumerate}
		\item The edge weights of $\mathcal{G}$ satisfies 
		$
			\mathcal{B} w = 0.
		$
		\item The Laplacian matrix of $\mathcal{G}$ satisfies 
		$
		\ker({L}) = \ker({L}^\top) = \mathrm{span}(\mathds{1}). 
		$
	\end{enumerate}
\end{lem}
The strong connectivity implies that there is only one zero eigenvalue of ${L}$ \cite{mesbahi2010graph}, and the balance of $\mathcal{G}$ then indicates that both the row and column sums of ${L}$ are zero.

\begin{rem}
	Undirected graphs can be viewed as special balanced directed graphs with bidirectional edges. The Laplacian matrix of an undirected graph is 
	$
		{L} = \mathcal{B} W \mathcal{B}^\top,
	$
	where $\mathcal{B}$ is an incidence matrix obtained by assigning an arbitrary orientation to each edge of the undirected graph, and $W$ is a positive diagonal matrix representing edge weights. 
\end{rem}

Next, we recapitulate the notion of graph clustering, whose concept can be found in e.g., \cite{Monshizadeh2014,ishizaki2015clustereddirected,Ishizaki2014,Besselink,jongsma2018model,XiaodongTAC2018MAS}.

\begin{defn} 
	Let $\mathcal{G}: = (\mathcal{V}, \mathcal{E})$ be a directed graph. Then, a graph clustering is a partition of $\mathcal{V}$ into $r$ nonempty disjoint subsets $ \mathcal{C}_1,\mathcal{C}_2,\cdots,\mathcal{C}_r$ covering all the
	elements in $\mathcal{V}$, where $\mathcal{C}_i$ is called a cluster of $\mathcal{G}$.
\end{defn}
Let $\{ \mathcal{C}_1,\mathcal{C}_2,\cdots,\mathcal{C}_r \}$ be a clustering of $\mathcal{G}$ with $n$ vertices. This graph clustering can be characterized by a binary characteristic matrix $\Pi \in \mathbb{R}^{n \times r}$, whose rows and columns are corresponding to the vertices and clusters, respectively:
	\begin{equation*}  
	\Pi_{ij} :=  \begin{cases}
	1 & \text{if vertex}~i \in \mathcal{C}_j,\\
	0 & \text{otherwise}. \\ 
	\end{cases}
	\end{equation*}

\begin{rem}
	Note that all the clusters are nonoverlapping, i.e., each vertex can be not assigned to distinct clusters. Therefore, each row of the characteristic matrix $\Pi$ only has one nonzero element. Specifically, we have 
	\begin{equation} \label{eq:Piproperty}
	\Pi \mathds{1}_r = \mathds{1}_n~\text{and}~\mathds{1}^\top_n \Pi=\left[|\mathcal{C}_1|, |\mathcal{C}_2|,\cdots,|\mathcal{C}_r| \right].
	\end{equation}
\end{rem}


\subsection{Problem Setup}

In this paper, we consider a network system evolving over a directed graph $\mathcal{G}$, which is simple, weighted and strongly connected. The dynamics of each vertex is governed by 
\begin{equation} \label{eq:diffcoupling}
\dot{x}_i(t) =  -\sum_{j=1}^{n} \mathcal{A}_{ij}  \left[x_i(t) - x_j(t)\right]  + \sum_{j=1}^{p} f_{ik} u_k(t),
\end{equation}
where $x_i(t) \in \mathbb{R}$ is the state of vertex $i$, and $\mathcal{A}_{ij}$ is the $(i,j)$-th entry of the adjacency matrix of $\mathcal{G}$, representing the strength of the coupling between vertices $i$ and $j$. $u_k(t) \in \mathbb{R}$ is the external input, and $f_{ij}\in \mathbb{R}$ is the gain of the $j$-th input acting on vertex $i$, which is zero if and only if $u_j$ has no effect on vertex $i$.
Let $F\in \mathbb{R}^{n \times p}$ be the matrix such that $F_{ij} = f_{ij}$. We then present the dynamics of the overall network in a compact form as
\begin{equation} \label{sysh}
\bm{\Sigma}:  
\begin{cases}
\dot{x}(t) = - {L} x(t) + {F}u(t),
\\
y(t) = Hx(t),
\end{cases}
\end{equation}
with $x(t) := \left[x_1, x_2, \cdots, x_n \right]^\top \in \mathbb{R}^{n}$ and $u := \left[{u}_1, u_2, \cdots, u_p \right]^\top\in \mathbb{R}^{p}$. The vector $y\in \mathbb{R}^{q}$ collects the outputs of the network, and $H$ is the output matrix. 


This paper aims for structure-preserving model reduction of diffusively coupled networks in form of \eqref{sysh}, and the reduced-order model not only approximates the input-output mapping of the original network system with a certain accuracy but also inherits an interconnection structure with diffusive  couplings. To this end, we adopt graph clustering to build up a reduced-order network model. 
Specifically, the problem addressed in this paper is as follows.
\begin{prob}\label{pro1}
 Given a network system $\bm{\Sigma}$ as in \eqref{sysh} and a graph clustering $\{ \mathcal{C}_1,\mathcal{C}_2,\cdots,\mathcal{C}_r \}$, find a reduced-order model
\begin{equation} \label{sysr}
\bm{\hat{\Sigma}}: 
\begin{cases} 
\dot{\hat{x}}  =  - \hat{ L} x + \hat{ F}u 
\\
\hat{y} =\hat{H} \hat{x}
\end{cases}
\end{equation}
with $\hat{x} \in \mathbb{R}^{r}$, $r \ll n$, 
such that $\hat{L}$ is the Laplacian matrix of a reduced directed graph, and the reduction error $\lVert \bm{\Sigma}-\bm{\hat{\Sigma}} \rVert_{\mathcal{H}_2}$ is minimized. $\hat{L} \in \mathbb{R}^{r \times r}$, $\hat{F} \in \mathbb{R}^{r \times p}$, $\hat{H} \in \mathbb{R}^{q \times r}$ are matrices depending on the graph clustering.   
\end{prob}

It is worth emphasizing that Problem~\ref{pro1} does not aim to find an appropriate graph clustering of the network $\mathcal{G}$. Instead, we focus on how to establish a ``good'' reduced-order model with given clusters. Thus, it is an essentially different problem from e.g., \cite{XiaodongTAC2018MAS,Xiaodong2020Digraph,Ishizaki2014,ishizaki2015clustereddirected}, and we do not apply the Petrov-Galerkin projection framework.

\section{Main Result}
\label{sec:Main}
In this section, a novel model reduction approach for network systems is presented with two steps. In the fist step, a parameterized model of a reduced network is constructed on the basis of graph clustering. Then, the second step computes a set of parameters in an optimal fashion such that the $\mathcal{H}_2$-norm of approximation error is minimized. 

\subsection{Parameterized Reduced-Order Network Model} 
\label{sec:Modeling}

Given a graph clustering of the original network, we present a parameterized model for the reduced network, whose interconnection topology is determined by the clustering. An important property of this parameterized model is that it guarantees the boundedness of the reduction error
$\lVert \bm{\Sigma}-\bm{\hat{\Sigma}} \rVert_{\mathcal{H}_2}$ for all positive edge weights.

To derive a parameterized reduced-order network model with such a property, we first convert the system \eqref{sysh} to its \textit{balanced graph representation} as follows. 
\begin{lem}
	If the underlying graph of $\bm{\Sigma}$ in \eqref{sysh} is strongly connected, then there exists a diagonal $M \in \mathbb{R}^{n \times n}$ with positive diagonal entries such that $\bm{\Sigma}$ is equivalent to
	\begin{equation} \label{sysbal}
	\left \{
	\begin{array}{l}
	M \dot{x}(t) = - {L_b} x(t) + F_bu(t),\\
	y(t) = H x(t),
	\end{array}
	\right.
	\end{equation} 
	where $F_b = M {F}$ and $L_b = M {L}$ is the Laplacian of a balanced graph.  
\end{lem} 

The proof follows directly from \cite{Xiaodong2020Digraph}. Next, we establish a reduced-order model using the representation \eqref{sysbal} to guarantee a bounded reduction error
$\lVert \bm{\Sigma}-\bm{\hat{\Sigma}} \rVert_{\mathcal{H}_2}$. 

Let $\mathcal G_b$ be the balanced graph of $\mathcal{G}$. Note that $\mathcal{G}$ and $\mathcal G_b$ have the same incidence matrix $\mathcal{B}$. Given a graph clustering $\{\mathcal{C}_1, \mathcal{C}_2, \cdots, \mathcal{C}_r\}$, the \textit{quotient graph} $\hat{\mathcal{G}}_b$ is $r$-vertex directed graph obtained by aggregating all the vertices in each cluster as a single vertex, while retaining connections between clusters and ignoring the edges within clusters. Specifically, if there is an edge $(i,j)$ in $\mathcal G_b$ with vertices $i,j$ in the same cluster, then it will not be presented as an edge in $\hat{\mathcal{G}}_b$. If there exists an edge $(i,j)$ with $i \in \mathcal{C}_k$ and $j \in \mathcal{C}_l$, then there is an edge $(k,l)$ in $\hat{\mathcal{G}}_b$. 

Let $\hat{\mathcal{B}}$ be the incidence matrix of the quotient graph  $\hat{\mathcal{G}}_b$. Algebraically, it can be verified that $\hat{\mathcal{B}}$ is obtained by removing all the zero columns of $\Pi^\top \mathcal{B}$, where $\mathcal{B}$ is the incidence matrix of $\mathcal G_b$ (or $\mathcal G$).
Furthermore, we denote 
\begin{equation}\label{omeg}
	\hat{W} = \diag(\hat{w}), \ \text{with} \ \hat{w} = \begin{bmatrix}
	\hat{w}_1 & \hat{w}_2 & \cdots & \hat{w}_m 
	\end{bmatrix}^\top, 
\end{equation} 
as the edge weight matrix of $\mathcal G_b$, where $\hat{w}_k \in \mathbb{R}_+$,
and $m$ is number of edges in $\hat{\mathcal{G}}_b$. In order to maintain $\hat{\mathcal{G}}_b$ as a balanced graph, we impose the constraint on its edge weights as 
\begin{equation} \label{eq:Bw0}
	\hat{\mathcal{B}} \hat{w} = 0,
\end{equation}
according to Lemma~\ref{lem:property}. Thereby, the dynamics on the balanced quotient graph $\hat{\mathcal{G}}_b$ is then obtained as 
\begin{equation} \label{sysbalr}
	\left \{
	\begin{array}{l}
	\hat{M} \dot{x}(t) = - \hat{L}_b(\hat{W}) x(t) + \hat F_b u(t),\\
	y(t) = \hat H x(t),
	\end{array}
	\right.
\end{equation} 
with the reduced matrices
\begin{align} \label{sysrMs}
	\hat{M} &= \Pi^\top M \Pi, \
	\hat{L}_b(\hat{W}) = \hat{\mathcal{B}}_0 \hat{W} \hat{\mathcal{B}}^\top, 
	\nonumber \\
	\hat{F}_b &= \Pi^\top F_b ,
	\ \text{and} \
	\hat{H} =  H \Pi, 
\end{align}
where $\hat{\mathcal{B}}_0$ is the binary matrix obtained by replacing all the
``$-1$" entries with zeros in $\hat{\mathcal{B}}$, and $\hat{M}_b \in \mathbb{R}^{r \times r}$, $\hat F_b \in \mathbb{R}^{r \times p}$, and $\hat{H} \in \mathbb{R}^{q \times r}$ 
are reduced matrices determined by the given clustering of $\mathcal G_b$. Since the graph clustering is given, i.e., $\Pi$ is known, the only parameters to be decided are the weights in $\hat{W}$, which satisfy the constraint \eqref{eq:Bw0}.  

From the reduced graph balanced representation \eqref{sysbalr}, we immediately construct a parameterized reduced-order model in the form of \eqref{sysr} with the reduced matrices
\begin{align} \label{sysrMs1}
\hat{L}(\hat{W}) = \hat{M}^{-1} \hat{\mathcal{B}}_0 \hat{W} \hat{\mathcal{B}}^\top, \ 
\hat{F} = \hat{M}^{-1} \Pi^\top F_b, \
\hat{H}  = H \Pi,
\end{align}
where $\hat{L}$ represents a reduced weighted graph $\hat{\mathcal{G}}$.
In \eqref{sysrMs1}, only the weight matrix $\hat{W}$ is to be determined, 
which is selected from
the following set
\begin{equation} \label{eq:Wconstraint}
\mathcal{M}: = \{ W = \diag(\hat{w})  \mid \hat{w} \in \mathbb{R}_+^m, \  \hat{\mathcal{B}} \hat{w} = 0 \}.
\end{equation}

In the following example, we demonstrate the parameterized modeling of a simplified dynamic network.  
\begin{exm}	
Consider an network example in vehicle formation \cite{fax2001graph}, where the formation topology $\mathcal{G}$ is depicted in Fig.~\ref{fig:Example6Nodes}. Clearly, $\mathcal{G}$ is balanced, i.e., $\mathcal{G} = \mathcal{G}_b$, with the incidence matrix 
\begin{equation*}
\mathcal{B} = \left[\scriptsize	{\begin{array}{cccccccccc}
1 & -1& 0 & 0 & 0 & 0 & 0 & 0 & -1& 0\\
-1& 1 & 1 &-1 & 0 & 0 & 0 & 0 & 0 & 0\\
0 & 0 & 0 & 1 & 1 & 1 & 0 &-1 & 0 & -1\\
0 & 0 & 0 & 0 &-1 & 0 & 1 & 0 & 0 & 0\\
0 & 0 & 0 & 0 & 0 &-1 &-1 & 1 & 0 & 0\\
0 & 0 & -1& 0 & 0 & 0 & 0 & 0 & 1 & 1\\
\end{array}} \right].
\end{equation*}
Suppose that each vehicle is modeled as a first-order integrator which has the identical mass, i.e., $M = I_6$. 
An external control $u$ is applied on vertex 4, and the vertex 1 is measured as the output signal $y$. Then, the network model is obtained in the form of \eqref{sysbal} with  
	\begin{align*}
	 & L_b= \left[\scriptsize {\begin{array}{cccccc}
		2 & -2 & 0 & 0 & 0 & 0\\
		-1 & 3 & 0 & 0 & 0 & -2\\
		0 & -1 & 4 & -2 & -1 & 0\\
		0 & 0 & 0 & 2 & -2 & 0\\
		0 & 0 & -3 & 0 & 3 & 0\\
		-1 & 0 & -1 & 0 & 0 & 2
	\end{array}} \right], 
	F_b  = 
	\left[\scriptsize {\begin{array}{c}
		0 \\ 0 \\ 0 \\ 1 \\ 0 \\ 0
	\end{array}} \right], 
	\\
	& \text{and} \ H = 
	\left[\scriptsize {\begin{array}{cccccc}
		1 & 0 & 0 & 0 & 0 & 0
	\end{array}} \right].	 
	\end{align*} 

	Consider a clustering of $\mathcal{G}_b$ as
	$\mathcal{C}_1 = \{1,2\},~ \mathcal{C}_2 = \{3,4,5\},~ \mathcal{C}_3 = \{6\}$, which leads to 
	the characterization matrix as 
	\begin{equation*} 
	\Pi = \begin{bmatrix}
	1 & 1 & 0 & 0 & 0 & 0\\
	0 & 0 & 1 & 1 & 1 & 0\\
	0 & 0 & 0 & 0 & 0 & 1\\
	\end{bmatrix}^\top
	\end{equation*}
	The topology of 
	the quotient graph $\hat{\mathcal{G}}_b$ is shown in Fig.~\ref{fig:Example4Nodes} with the incidence matrix   
	\begin{equation*}
	\hat{\mathcal{B}} = 
	\begin{bmatrix}
		1 & -1  & -1 &0 \\
		0 & 1  & 0 &-1\\
		-1 & 0& 1 &1\\
	\end{bmatrix}.
	\end{equation*}
	All the edge weights of $\hat{\mathcal{G}}_b$ are positive parameters to be determined, as labeled in Fig.~\ref{fig:Example4Nodes}, which leads to the parameterized Laplacian matrix as
	\begin{equation*}
	\hat{L}_b (\hat{W}) = \begin{bmatrix}
	\hat{w}_1  &  0  & - \hat{w}_1 \\
	-\hat{w}_2 & \hat{w}_2  & 0 \\
	- \hat{w}_3 & -\hat{w}_4 & \hat{w}_3 + \hat{w}_4
	\end{bmatrix}.
	\end{equation*}
	The weights satisfy the constraint $\hat{\mathcal{B}} \hat{w} = 0$,
	namely, $\hat{w}_3 = \hat{w}_1 - \hat{w}_2$, $\hat{w}_4 = \hat{w}_2$,
	such that $\hat{\mathcal{G}}_b$ is balanced.
	The other matrices in the reduced-order model \eqref{sysbalr} are computed as 
	\begin{align*}
		&\hat{M} = \Pi^\top M \Pi = \begin{bmatrix}
		2 & 0 & 0 \\0 & 3 & 0\\0 & 0 & 1
		\end{bmatrix},
		\hat{F}_b = \Pi^\top F_b =
		\begin{bmatrix}
		0 \\ 1 \\ 0 
		\end{bmatrix}, \\ 
		&\text{and} \
		\hat{H}  = H \Pi =
		\begin{bmatrix}
		1 & 0 & 0 
		\end{bmatrix}.
	\end{align*}
	Then, in the parameterized reduced-order model \ref{sysr}, we have 
	\begin{equation*}
		\hat{L}(\hat{W}) = \hat{M}^{-1} \hat{L}_b (\hat{W}) = \begin{bmatrix}
		\frac{1}{2}\hat{w}_1  &  0  & -\frac{1}{2} \hat{w}_1 \\
		-\frac{1}{3}\hat{w}_2 & \frac{1}{3}\hat{w}_2  & 0 \\
		- \hat{w}_3 & -\hat{w}_4 & \hat{w}_3 + \hat{w}_4
		\end{bmatrix},
	\end{equation*}
	with $\hat{w}_3 = \hat{w}_1 - \hat{w}_2$, and $\hat{w}_4 = \hat{w}_2$. The corresponding reduced graph is depicted in Fig.~\ref{fig:Example4Nodes1}, which is no longer balanced.
	
		\begin{figure}[!tp]
		\begin{minipage}[t]{\linewidth}
			\centering
			\includegraphics[width=0.5\textwidth]{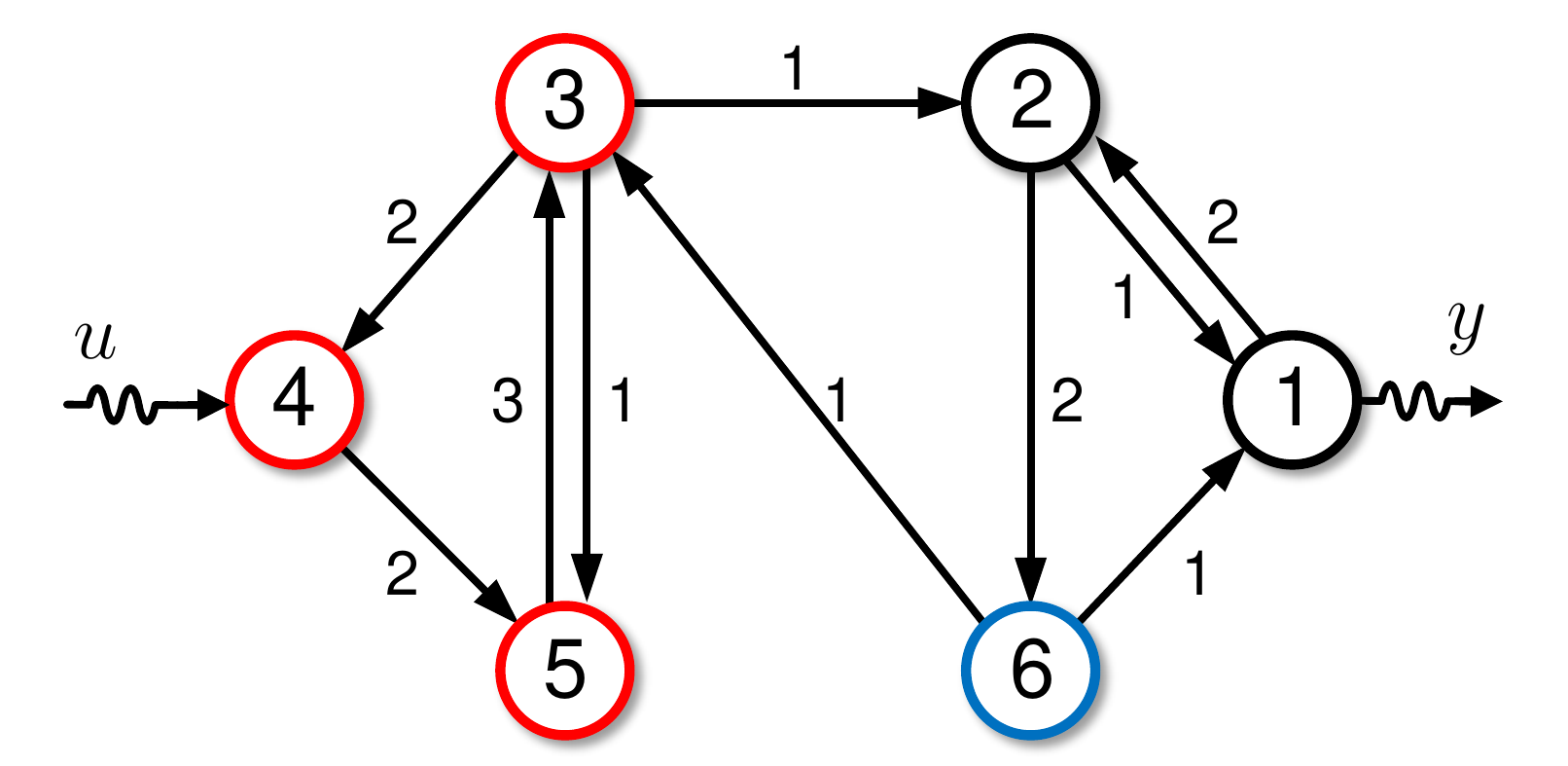}    
			\subcaption{}
			\label{fig:Example6Nodes}
		\end{minipage}%
		\hfill
		\centering
		\begin{minipage}[t]{0.35\linewidth}
			\centering
			\includegraphics[width=\textwidth]{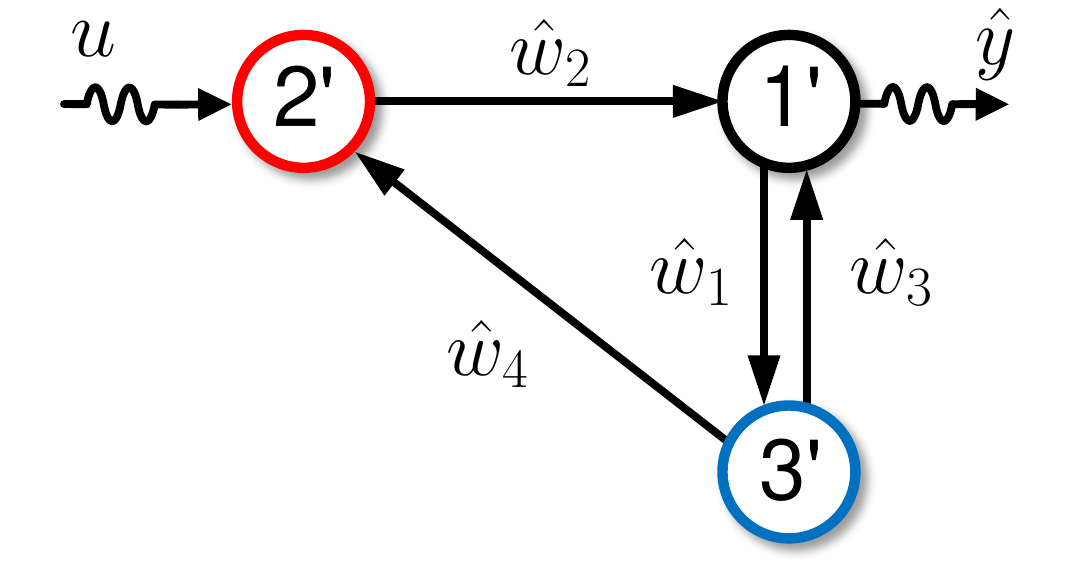}    
			\subcaption{}
			\label{fig:Example4Nodes}
		\end{minipage}%
		\begin{minipage}[t]{0.35\linewidth}
			\centering
			\includegraphics[width=\textwidth]{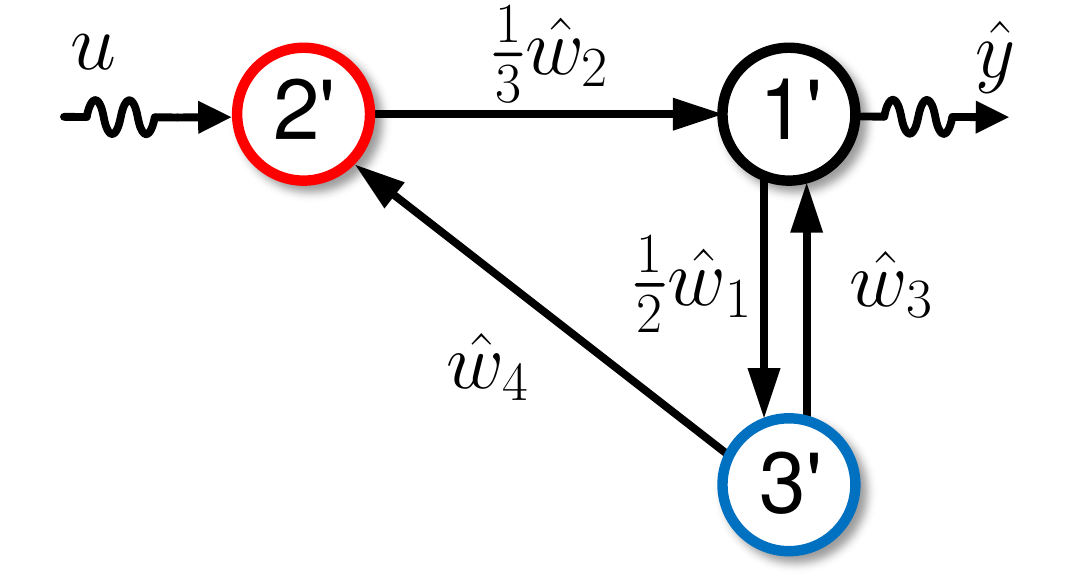}    
			\subcaption{}
			\label{fig:Example4Nodes1}
		\end{minipage}%
		\caption{(a) A directed balanced network $\mathcal{G}_b$ consisting of 6 vertices, in which vertex 3 is controlled and vertex 4 is measured. The vertices in different clusters are indicated by distinct colors. (b) The quotient graph $\hat{\mathcal{G}}_b$ consisting of 3 vertices, where the edge weights are parameters to be determined. The quotient graph is balanced, when the constraints $\hat{w}_3 = \hat{w}_1 - \hat{w}_2$ and $\hat{w}_4 = \hat{w}_2$ are imposed. (c) The resulting reduced graph $\hat{\mathcal{G}}$.}
	\end{figure} 
\end{exm}

\begin{rem}
	The physical interpretation of the reduced matrices in \eqref{sysrMs} are explained. $\hat{M}$ is constructed such that the mass of vertex $k$ in $\hat{\mathcal{G}}_b$ is equal to the mass sum of  all the vertices in $\mathcal{C}_k$ in $\mathcal{G}_b$. The expression of $\hat{F}_b$ means that if a vertex in a cluster $\mathcal{C}_k$ of  $\mathcal{G}_b$ is controlled by an external input, then vertex $k$ in $\hat{\mathcal{G}}_b$ is also controlled. Analogously, $\hat{H}$ indicates that a vertex $k$ in $\hat{\mathcal{G}}_b$ is measured if there is a measurement taken from a vertex in $\mathcal{C}_k$.  
\end{rem}	

With the reduced matrices in \eqref{sysrMs1} and the constraint in \eqref{eq:Bw0}, an important property of the reduced-order network model $\bm{\hat{\Sigma}}$ is that it guarantees the $\mathcal{H}_2$ reduction error between the original system $\bm{\Sigma}$ in \eqref{sysh} and $\bm{\hat{\Sigma}}$ is always bounded.

	The computation of the reduction error amounts to find the $\mathcal{H}_2$ norm of the following {error system}:
\begin{equation} \label{syserr}
G_e(s)=\mathcal{C}_e(sI-\mathcal{A}_e)^{-1}\mathcal{B}_e,
\end{equation} 
where
\begin{align*} 
\mathcal{A}_e  = -\begin{bmatrix} L & 0 \\ 0 & \hat{L} \end{bmatrix}, \
\mathcal{B}_e  =\begin{bmatrix} F \\ \hat F \end{bmatrix},\ \mathcal{C}_e=\begin{bmatrix} H &  -\hat{H} \end{bmatrix}.
\end{align*}
Note that $\mathcal{A}_e$ is not Hurwitz, since ${L}$ and $\hat{L}$ are both Laplacian matrices containing zero eigenvalues. Thus, $\| G_e(s) \|_{\mathcal{H}_2}$ cannot be calculated directly using the state space representation \eqref{syserr}. 
Here, we employ the following matrices
\begin{equation} \label{eq:S}
S_n =  \begin{bmatrix}
-I_{n-1}  \\
\mathds{1}_{n-1}^\top  
\end{bmatrix} 
\in \mathbb{R}^{n \times (n-1)},
S_r =  \begin{bmatrix}
-I_{r-1}  \\
\mathds{1}_{r-1}^\top   
\end{bmatrix} \in \mathbb{R}^{r \times (r-1)}
\end{equation}
which are independent of system dynamics and satisfy 
\begin{equation*}
	S_n^\top \mathds{1}_n = 0, \ \text{and} \ S_r^\top \mathds{1}_r = 0.
\end{equation*}
Let their left pseudo inverses be
\begin{align*}
S_n^+ : &= (S_n^\top M^{-1} S_n)^{-1} S_n^\top M^{-1} \in \mathbb{R}^{(n-1) \times n}, 
\\ 
S_r^+ : &= (S_r^\top \hat{M}^{-1} S_r)^{-1} S_r^\top \hat{M}^{-1}  \in \mathbb{R}^{(r-1) \times r} .
\end{align*}
Then, using the matrices in \eqref{eq:S}, we show the following result.

\begin{lem}\label{lem:property}
	Consider the network system $\bm \Sigma$ in \eqref{sysh} and the reduced-order network model $\bm{\hat{\Sigma}}$ in \eqref{sysr} with matrices in \eqref{sysrMs1}. Then, 
	 $\bm{\Sigma} - \hat{\bm{\Sigma}} \in \mathcal{H}_2$ holds for all $\hat{W} \in \mathcal{M}$.
\end{lem}
\begin{proof}
	With $S_n$ and $S_r$ in \eqref{eq:S}, we construct a nonsingular $(n+r) \times (n+r)$ matrix as 
	\begin{equation} \label{eq:Ue}
	\mathcal{U}_e = 
	\begin{bmatrix}
	\sigma_M^{-1} \mathds{1}_n & 0 &  M^{-1}S_n & 0\\
	0 & \sigma_M^{-1} \mathds{1}_r   & 0 &  \hat{M}^{-1}S_r\\
	\end{bmatrix}, 
	\end{equation}
	where $\sigma_M := \mathds{1}_n^\top M \mathds{1}_n = \mathds{1}_r^\top \Pi^\top  M  \Pi \mathds{1}_r = \mathds{1}_r^\top \hat{M} \mathds{1}_r$.
	The inverse of $\mathcal{U}_e$ is given as
	\begin{equation*}
	\mathcal{U}_e^{-1} = \begin{bmatrix}
	\mathds{1}_n^\top M & 0\\0 & \mathds{1}_r^\top \hat{M}  \\
	S_n^\dagger M  & 0 \\0 & S_r^\dagger \hat{M}
	\end{bmatrix}.
	\end{equation*}
	Note that $L = M^{-1} L_b$ and $\hat{L} = \hat{M}^{-1} \hat{L}_b$, where both $L_b$ and $\hat{L}_b$ are the Laplacian matrices of balanced graphs, satisfying $\mathds{1}_n^\top L_b = 0$, $L_b \mathds{1}_n = 0$, and $\mathds{1}_r^\top  \hat{L}_b = 0$, $\hat L_b \mathds{1}_r = 0$. 
	Using these properties, we obtain
	\begin{align} \label{eq:Ge0}
	 G_e(s) 
	 \nonumber 
	 &= \mathcal C_e\mathcal{U}_e  (sI- \mathcal{U}_e^{-1} \mathcal A_e \mathcal{U}_e)^{-1} \mathcal{U}_e^{-1} \mathcal B_e \\
	 & = \begin{bmatrix}
	 \bar{C}_e & C_e
	 \end{bmatrix} 
	 \begin{bmatrix}
	 0 & 0 \\
	 0 & (sI - A_e)^{-1}
	 \end{bmatrix}
	 \begin{bmatrix}
	 \bar{B}_e\\ B_e
	 \end{bmatrix}, \nonumber
	 \\
	 & = \bar{C}_e \bar{B}_e + C_e (sI - A_e)^{-1} B_e,
	\end{align}
	where
	\begin{align} \label{eq:AeBeCe}	
		A_e & = - 
		\begin{bmatrix}
		S_n^+ L_b M^{-1} S_n & 0 \\ 0 & S_r^+ \hat{L}_b  \hat{M}^{-1} S_r
		\end{bmatrix}, 
		\nonumber\\
		B_e & =  \begin{bmatrix} S_n^+ F_b \\ S_r^+ \hat{F}_b \end{bmatrix}, \ 
		C_e=\begin{bmatrix} H M^{-1}S_n &  -\hat{H} \hat{M}^{-1}S_r \end{bmatrix},
	\end{align}
	and 
	\begin{equation}
	\bar{B}_e : = \begin{bmatrix}
	\mathds{1}_n^\top F_b \\  \mathds{1}_r^\top \hat F_b 
	\end{bmatrix}, \ \text{and} \
	\bar{C}_e : = \sigma_M^{-1} \begin{bmatrix}
	H \mathds{1}_n &  - \hat{H} \mathds{1}_r 
	\end{bmatrix}.
	\end{equation}
	It follows from \eqref{eq:Piproperty} that $\mathds{1}_r^\top \hat{F}_b = \mathds{1}_n^\top {F}_b$, and $\hat{H} \mathds{1}_r = {H} \mathds{1}_n$, which yield
	$
		\bar{C}_e  \bar{B}_e  = 
		\sigma_M^{-1} (H \mathds{1}_n\mathds{1}_n^\top F_b  - \hat{H} \mathds{1}_r \mathds{1}_r^\top \hat F_b) = 0.
	$
	Thus, \eqref{eq:Ge0} becomes 
	\begin{equation} \label{eq:Ge}
		G_e(s) = C_e (sI - A_e)^{-1} B_e.
	\end{equation} 
	It is not hard to verify that both the matrices 
	$ - S_n^+ L_b M^{-1} S_n$ and $ - S_r^+ \hat{L}_b \hat{M}^{-1} S_r$ are Hurwitz. Consequently, $G_e(s)$ in \eqref{eq:Ge} is  asymptotically stable, i.e., $\bm{\Sigma} - \hat{\bm{\Sigma}} \in \mathcal{H}_2$, for all $\hat{W} \in \mathcal{M}$.
\end{proof}

Next, we discuss the consensus property of the reduced-order network \eqref{sysr} with the matrices in \eqref{sysrMs1}. Consensus is a typical property of diffusively coupled networks, and it implies the nodal states converge to a common value in the absence of the external input. More precisely, the network system in \eqref{sysh} reaches consensus if 
\begin{equation*}
\lim_{t \rightarrow \infty} [x_i(t) - x_j(t)] = 0
\end{equation*}
holds for all $i,j\in \mathcal{V}$ and all initial conditions. 

\begin{pro}
	Consider the network system $\bm{\Sigma}$ in \eqref{sysh} which reaches consensus. Then, the reduced-order model $\bm{\hat{\Sigma}}$ in \eqref{sysr} also reaches consensus, for any  clustering $\Pi$ and $\hat{W} \in \mathcal{M}$.
\end{pro}
\begin{proof}
	It can verified that the parameterized Laplacian matrix $\hat{L}$ defined in \eqref{sysrMs1} characterizes a strongly connected graph. Thus, $\hat{L}$ has only one zero eigenvalue. Then, the proof immediately follows from \cite{mesbahi2010graph,Xiaodong2020Digraph}. 
\end{proof}

The parameterized modeling of the reduced dynamic network using the graph balanced representation in \eqref{sysbal} guarantees the stability of the error system \eqref{syserr}, whose $\mathcal{H}_2$ norm can be evaluated via the transfer function \eqref{eq:Ge} with the Hurwitz matrix $A_e$. Note that in \eqref{eq:Ge}, the matrices $S_n$ and $S_r$ in \eqref{eq:S} is only dependent on the sizes of the networks, and $\Pi$ is known for a given graph clustering, then the weights in $\hat{W}$ become the only unknown parameters to be determined in the follow-up procedure. In the following section, we aim for an optimal selection of the edges weights in the reduced network. 
	
%
%
%
%
 

\subsection{Optimal Edge Weighting}
\label{sec:weight}

In this section, we aim for an optimization scheme for determining  $\hat{W} \in \mathcal{M}$ that minimizes the approximation error $\lVert G_e(s)\rVert_{\mathcal{H}_2}$. Thereby, the following problem is addressed.
\begin{prob}\label{prob2}
	Consider the original network system $\bm\Sigma$ in \eqref{sysh}. Given a graph clustering $\Pi$, find a $\hat{W} \in \mathcal{M}$ such that $\lVert G_e(s)\rVert_{\mathcal{H}_2}$ is minimized, where $\bm{\hat{\Sigma}}$ is the reduced network model defined in \eqref{sysr} with the matrices \eqref{sysrMs1}. 
\end{prob}


To solve this problem, we apply an optimization technique based on the \textit{convex-concave decomposition}, which can be implemented to search for a set of optimal weights iteratively. A fundamental step toward the implementation is to develop a necessary and sufficient condition for characterizing $\lVert G_e(s)\rVert_{\mathcal{H}_2}$, which leads to  suitable constraints for the optimization problem.

\begin{thm}\label{theorem-1}
Given the network system $\bm{\Sigma}$ in \eqref{sysh}. There exists a reduced-order network model $\bm{\hat{\Sigma}}$ in \eqref{sysr} such that $\|G_e(s)\|_{\mathcal{H}_2}^2<\hat{\gamma}$ if and only if there exist matrices  $\hat{Q}=\hat{Q}^\top>0$ with dimension $\hat Q\in\mathbb{R}^{(n+r-2)\times(n+r-2)}$,  $\hat{R}=\hat{R}^\top>0$ with dimension $R\in\mathbb{R}^{q\times q}$, $\hat{W}\in\mathcal{M}$, and $\hat{\delta} \in \mathbb{R}_+$, such that the following inequalities are satisfied,
{\small\begin{eqnarray}
\label{inquali-1}
\begin{bmatrix}
\hat{Q}\bar{A}+\bar{A}^\top\hat{Q}& \hat{Q}B_e& \hat{Q}E\\
B^\top_{e}\hat{Q} & -\hat{\delta}I& 0\\
E^\top\hat{Q}& 0& 0
\end{bmatrix}+\begin{bmatrix}
-\bar{A}^\top_{r}\bar{A}_{r}& 0& \bar{A}^\top_{r}\\
0& 0& 0\\
\bar{A}_{r}& 0 & -I
\end{bmatrix}<0,
\\
	\begin{bmatrix}\label{inquali-2}
	\hat{Q} & \hat{\delta}C^\top_e\\
	\hat{\delta}C_e & \hat{R} \\
	\end{bmatrix}>0, 
	\\
	\tr(\hat{R})<\hat{\gamma}, 
	\label{inquali-3}
	\end{eqnarray}}%
where $B_e$, $C_e$ are defined in \eqref{eq:AeBeCe}, and 
\begin{align}\label{decouple:Ae}
  \bar{A}&=\begin{bmatrix}
  -S_n^+ L_b M^{-1} S_n& 0\\
  0& 0
  \end{bmatrix}, \
  E =\begin{bmatrix}
  0& 0\\
  I& 0
  \end{bmatrix}
  \nonumber\\ 
  \bar{A}_r & =\begin{bmatrix}
   0& -S_r^+ \hat{\mathcal{B}}_0 \hat{W} \hat{\mathcal{B}}^\top\hat{M}^{-1} S_r\\
  0& 0
  \end{bmatrix}.
\end{align}
\end{thm}

\begin{proof}
Consider the error system $G_e(s)$ in \eqref{eq:Ge}, which is asymptotically stable. Following e.g., \cite{pipeleers2009extended}, we have $\|G_e(s)\|_{\mathcal{H}_2}^2<\gamma$, with $\gamma \in \mathbb{R}_+$, if and only if there exist matrices $Q=Q^\top>0$ and $R=R^\top>0$ such that  
	\begin{eqnarray}
	\begin{bmatrix} \label{eq:ineq1}
	Q A_e+A^\top_e Q & QB_e \\
	B^\top_e Q & -I_p \\
	\end{bmatrix}
	&<&0, \\
	\begin{bmatrix}\label{eq:ineq2}
	Q & C^\top_e\\
	C_e & R \\
	\end{bmatrix}&>&0, \\
	\tr(R)&<&\gamma, \label{eq:ineq3}
	\end{eqnarray}
	where $A_e$, $B_e$, $C_e$ are defined in \eqref{eq:AeBeCe}.
  
In the following, we prove that the three inequalities are equivalent to \eqref{inquali-1}, \eqref{inquali-2}, and \eqref{inquali-3}, respectively. First, it is not hard to verify that \eqref{eq:ineq1} is equivalent to 
\begin{equation}\label{inequality:proof-1}
 \begin{bmatrix}
    QA_e+A^\top_{e}Q& QB_e& Q E\\
    B^\top_{e}Q& -I& 0\\
    E^\top Q& 0& -\delta I
\end{bmatrix}<0
\end{equation}
for a sufficiently large scalar $\delta \in \mathbb{R}_+$, where $E$ is defined in \eqref{decouple:Ae}.
Consider a nonsingular matrix 
\begin{equation*}
    T=\begin{bmatrix}
    I& 0& 0\\
    0&I &0\\
    -\bar{A}_{r}& 0& I
    \end{bmatrix}.
\end{equation*}
Pre- and post-multiplying by $T^\top$ and $T$, respectively, \eqref{inequality:proof-1} then becomes \eqref{inquali-1}, where the equation $A_e=\bar{A}+E\bar{A}_{r}$, and the substitutions $\hat{\delta}=\frac{1}{\delta}>0$, $\hat{Q}=\frac{1}{\delta}Q>0$ are used. 

Next, we observe that the following implications hold.
\begin{equation*}
        \begin{bmatrix}
          Q& C^{\top}_{e}\\
          C_{e}& R
        \end{bmatrix}>0
 \Leftrightarrow\begin{bmatrix}
 \frac{1}{\hat{\delta}}\hat{Q}& C^{\top}_{e}\\
  C_{e}& R
 \end{bmatrix}>0 
\Leftrightarrow\begin{bmatrix}
 \hat{Q}& \hat{\delta}C^{\top}_{e}\\
 \hat{\delta}C_{e}& \hat{R}
\end{bmatrix}>0
\end{equation*}
\begin{equation*}
    \tr(R)<\gamma\Leftrightarrow
    \frac{1}{\hat{\delta}}\tr(\hat{R})<\gamma
    \Leftrightarrow\tr(\hat{R})<\hat{\gamma},
\end{equation*}
with $\hat{R}=\hat{\delta}R$ and $\hat{\gamma}=\hat{\delta}\gamma$. As a result,  \eqref{inquali-2} and  \eqref{inquali-3} are equivalent to \eqref{eq:ineq2} and \eqref{eq:ineq3}, respectively. 
\end{proof}

Based on Theorem \ref{theorem-1}, we reformulate
Problem \ref{prob2} more explicitly as the following minimization problem 
\begin{align}\label{pro:min-1}
      &\min_{\hat{Q}>0, ~\hat{W}\in\mathcal{M}} \tr(R) 
      \\
  & \quad \quad  \text{s.t.\quad  \eqref{inquali-1} and  \eqref{inquali-2} hold},\nonumber
\end{align}
where $\hat{R}=\hat{\delta}R$ with a given $\hat{\delta} \in \mathbb{R}_+$. Note that the constraint \eqref{inquali-2} can be solved efficiently using standard LMI solvers, while \eqref{inquali-1}, due to the nonlinearity term $\bar{A}^\top_{r}\bar{A}_{r}$, is a bilinear matrix inequality, which causes the major challenge in solving the problem \eqref{pro:min-1}.

To handle the bilinear constraint \eqref{inquali-1}, we
adopt the technique called \textit{psd-convex-concave decomposition} \cite{dinh2011combining}.
\begin{defn} \label{def:psdcc}
A matrix-valued mapping $\Phi: \mathbb{R}^n \rightarrow \mathcal{S}^\ell$ is called positive semidefinite convex concave (psd-convex-concave) if $\Phi$ can be expressed as  
	$\Phi = \Phi_1 - \Phi_2$,
	where $\Phi_k$, with $k = 1, 2$, are positive semidefinite convex (psd-convex), i.e.,
	\begin{equation} \label{eq:psdc}
		\Phi_k (\lambda w_1 + (1 - \lambda) w_1) \leq \lambda \Phi_k (w_1) + (1 - \lambda) \Phi_k (w_2),
	\end{equation}
	holds for all $\lambda \in [0, 1]$ and $w_1, w_2 \in \mathbb{R}^n$.  The pair $(\Phi_1, \Phi_2)$ is called a psd-convex-concave
	decomposition of $\Phi$.
\end{defn}

Consider the bilinear inequality \eqref{inquali-1}, and define the following matrix-valued mapping:
\begin{equation}\label{mapping}
 \Phi(\hat{Q}, \hat{\delta}, \hat{W})=\psi(\hat{Q}, \hat{\delta})+\varphi(\hat{W}), 
\end{equation}
where 
\begin{align}
\psi(\hat{Q}, \hat{\delta})&=\begin{bmatrix}
\hat{Q}\bar{A}+\bar{A}^\top\hat{Q}& \hat{Q}B_e& \hat{Q}E\\
B^\top_{e}\hat{Q} & -\hat{\delta}I& 0\\
E^\top\hat{Q}& 0& 0
\end{bmatrix}, 
\\
\label{varphiW}
\varphi(\hat{W})&=\begin{bmatrix} 
-\bar{A}^\top_{r}\bar{A}_{r}& 0& \bar{A}^\top_{r}\\
0& 0& 0\\
\bar{A}_{r}& 0 & -I
\end{bmatrix}. 
\end{align}
Then, the following lemma shows that the pair $(\psi, -\varphi)$ is a psd-convex-concave decomposition of  $\Phi$. 
\begin{lem}\label{lem:psd-convex}
The matrix-valued mapping $\Phi(\hat{Q}, \hat{\delta}, \hat{W})$ in \eqref{mapping} is psd-convex-concave.
\end{lem}
\begin{proof}
Note that the  matrix $\psi(\hat{Q}, \hat{\delta})$ in \eqref{mapping} is linear with respect to $\hat{Q}$ and $\hat{\delta}$. Thus, it is immediate that $\psi(\hat{Q}, \hat{\delta})$ is psd-convex. Then, the claim holds if $-\varphi(\hat{W})$ in \eqref{mapping} is psd-convex.

With the structure of $\bar{A}_{r}$ in \eqref{decouple:Ae}, the only nonlinear submatrix in $-\varphi(\hat{W})$ can be expressed as
\begin{equation}\label{phi_1}
\bar{A}^\top_{r}\bar{A}_{r} =\begin{bmatrix}
0& 0\\
0& {\varphi}_{a}(\hat{W})
\end{bmatrix},   
\end{equation}
with 
$ 
{\varphi}_{a}(\hat{W}):=S_{r}^\top\hat{M}^{-1}\hat{\mathcal{B}}\hat{W}\hat{\mathcal{B}}_0^\top(S_r^+)^\top S_r^+\hat{\mathcal{B}}_0\hat{W}\hat{\mathcal{B}}^\top\hat{M}^{-1}S_r.
$ 
Then, showing the psd-convexity of $-\varphi(\hat{W})$ in \eqref{varphiW} is equivalent to prove that ${\varphi}_{a}(\hat{W})$ is psd-convex. 

Let $W_1, W_2 \in \mathcal{M}$, and denote $W_\lambda = \lambda\hat{W}_1+(1-\lambda)\hat{W}_2$. For any $\lambda\in[0, 1]$, we have   
\begin{align}\label{b}
&\varphi_{a}(W_\lambda)-\lambda \varphi_{a}(\hat{W}_1)-(1-\lambda)\varphi_{a}(\hat{W}_2)\nonumber \\
=&S_{r}^\top\hat{M}^{-1}\hat{\mathcal{B}} W_\lambda \hat{\mathcal{E}}^\top(S_r^+)^\top S_r^+\hat{\mathcal{E}}W_\lambda\hat{\mathcal{B}}^\top\hat{M}^{-1}S_r
\nonumber\\
&-t S_{r}^\top\hat{M}^{-1}\hat{\mathcal{B}}\hat{W}_1\hat{\mathcal{E}}^\top(S_r^+)^\top S_r^+\hat{\mathcal{E}}\hat{W}_1\hat{\mathcal{B}}^\top\hat{M}^{-1}S_r
\nonumber\\
&-(1-\lambda) S_{r}^\top\hat{M}^{-1}\hat{\mathcal{B}}\hat{W}_2\hat{\mathcal{E}}^\top(S_r^+)^\top S_r^+\hat{\mathcal{E}}\hat{W}_2\hat{\mathcal{B}}^\top\hat{M}^{-1}S_r
\nonumber\\
=&-\lambda(1-\lambda)(V_{1}(\hat{W}_1-\hat{W}_2)V_{2}(\hat{W}_1-\hat{W}_2)V_{1}^\top)\leq0,
\end{align}
where $V_{1}=S_{r}^\top\hat{M}^{-1}\hat{\mathcal{B}}$, and $V_{2}= \hat{\mathcal{B}}_0^\top(S_r^+)^\top S_r^+\hat{\mathcal{B}}_0$. Since $-t(1-t)\leq0$ and
$
 V_{1}(\hat{W}_1-\hat{W}_2)V_{2}(\hat{W}_1-\hat{W}_2)V_{1}^\top\geq0,  
$
it holds that  
\begin{equation*}
\varphi_{a}(\lambda \hat{W}_1+(1-\lambda)\hat{W}_2)\leq \lambda \varphi_{a}(\hat{W}_1)+(1-\lambda)\varphi_{a}(\hat{W}_2),
\end{equation*}
which implies that the mapping $\varphi_{a}(\hat{W})$ is psd-convex from \eqref{eq:psdc}, i.e., $-\varphi(\hat{W})$ is psd-convex. As a result, it follows from Definition~\ref{def:psdcc} that the matrix-valued mapping $\Phi(\hat{Q}, \hat{\delta}, \hat{W})$ in \eqref{mapping} is psd-convex-concave.
\end{proof}

The psd-convex-concave decomposition in \eqref{mapping} allows us to linearize the optimization problem \eqref{pro:min-1} at a stationary point $\hat{W} \in \mathcal{M}$. To simplify the optimization procedure, we introduce a new optimization variable $\mu$ to eliminate the equality constraint $\hat{\mathcal{B}} \hat{w} = 0$ in \eqref{eq:Wconstraint}, where 
$\hat{W} = \diag(\hat{w})$, and $m$ is the number of edges in the reduced network.

Let $\bar{\mathcal{B}} \in \mathbb{R}^{\bar{r} \times m}$ be a full row rank matrix obtained by removing linearly independent rows of the $\hat{\mathcal{B}} \in \mathbb{R}^{r \times m}$, and it still holds that $\bar{\mathcal{B}} \hat{w} = 0$. Then, there exists a column permutation matrix $P \in \mathbb{R}^{m \times m}$ such that
\begin{equation*}
	\bar{\mathcal{B}} \hat{w} = \begin{bmatrix}
	\bar{\mathcal{B}}_a & \bar{\mathcal{B}}_b
	\end{bmatrix} P \hat{w} = \bar{\mathcal{B}}_a  \mu_a + \bar{\mathcal{B}}_b {\mu} = 0, \ \text{with} \
	P \hat{w} = \begin{bmatrix}
	\mu_a \\ \mu
	\end{bmatrix},
\end{equation*}
where $\bar{\mathcal{B}}_a \in \mathbb{R}^{\bar{r} \times \bar{r}}$ is full rank. $\mu \in \mathbb{R}_+^{\bar{m}}$, and $\bar{m} =  m - \bar{r}$, is defined as the new optimization variable. Note that 
\begin{equation} \label{eq:mu}
	\hat{w} = P^\top \begin{bmatrix}
		-\bar{\mathcal{B}}_a^{-1}\bar{\mathcal{B}}_b \\ I_{\bar{m}}
	\end{bmatrix} \mu, 
\end{equation}
which projects the weights $\hat{w}$ into $\ker(\hat{\mathcal{B}})$. Thereby, we rewrite the constraint $\hat{W} \in \mathcal{M}$ as $\mu \in \mathbb{R}_+^{\bar{m}}$ \eqref{eq:mu}.  Now, we redefine the matrix-valued mapping $\varphi(\hat{W})$ in \eqref{mapping} as 
\begin{equation} \label{eq:phi-chi}
\phi(\mu) = \varphi(\hat{W}),
\end{equation}
which remains psd-convex due to the linear relation in \eqref{eq:mu}.
The derivative of the matrix-valued mapping $\phi(\mu)$ at $\mu$ is a linear
mapping $D \phi:  \mathbb{R}_+^{\bar{m}} \rightarrow \mathcal{S}^{\ell}$, with $\ell = n+2r+p+pq-3$, which is defined as
\begin{equation}\label{eq:deriva}
D \phi(\mu) [h] =\sum_{i=1}^{\bar{m}} h_{i}\frac{\partial\phi}{\partial \mu_{i}}(\mu), \ \forall~h \in \mathbb{R}^{\bar{m}}.
\end{equation}

Given a point $\mu^{(k)} \in \mathbb{R}_+^{\bar{m}}$, the linearized formulation of the problem \eqref{pro:min-1} at $\mu^{(k)}$ is formulated as
\begin{align}\label{pro:lin-min1}
&\min_{\hat{Q}>0, \mu \in \mathbb{R}_+^{\bar{m}} }  f(\mu)=\tr(R)
\\
&\quad \text{s.t.} \quad 	\begin{bmatrix}
\hat{Q} & \hat{\delta}C^\top_e\\
\hat{\delta}C_e & \hat{R} \\
\end{bmatrix}>0,\ \hat{\delta} \in \mathbb{R}_+, \ \hat{R} = \hat{\delta} R >0   \nonumber
	\\
& \hspace*{35pt} \psi(\hat{Q}, \hat{\delta})+\varphi(\hat{W}^{(k)})+D\phi(\mu^{(k)})[\mu-\mu^{(k)}]<0, \nonumber
\end{align}
where the derivative of $\phi(\mu^{(k)})$ is given as
\begin{equation*}
	D\phi(\mu^{(k)})[\mu-\mu^{(k)}]:=\sum_{i=1}^{m}(\mu_{i}-\mu^{(k)}_{i})\frac{\partial\phi}{\partial \mu_{i}^{(k)}}(\mu^{(k)}),
\end{equation*}
 with 
$
 j=1, \cdots, \bar{m}.
$ 
Notice that 
the optimization problem \eqref{pro:lin-min1} is \textit{convex}, of which the global optimum can be solved efficiently using standard SDP solvers e.g., SeDuMi \cite{sturm1999SeDuMi}. Based on Lemma~\ref{lem:psd-convex} and \eqref{pro:lin-min1}, we are now ready to present an algorithmic approach for
solving the minimization problem in \eqref{pro:min-1} in an iterative fashion, see Algorithm~\ref{alg}, in which $\varepsilon\in \mathbb{R}_+$ is a prefixed error tolerance determining whether to terminate the iteration loop.
 
\begin{algorithm}[!tp]
	\caption{Iterative Edge Weighting}
\begin{algorithmic}[1]
\REQUIRE
     $L$, $F$, $H$, $\Pi$, and a small scalar $\hat{\delta}\in \mathbb{R}_+$
\ENSURE
       $\hat{W}^{*}$.

\STATE Compute the incidence matrix $\hat{\mathcal{B}}$ of the quotient graph $\hat{\mathcal{G}}_b$.

\STATE Choose an initial vector $\mu^{(0)} \in \mathbb{R}_+^{\bar{m}}$.

\STATE Set iteration step: $k \leftarrow 0$.
\REPEAT
\STATE Solve \eqref{pro:lin-min1} to obtain the optimal solution $\mu^{*}$.

\STATE   $k \leftarrow k+1$, and $\mu^{(k)} \leftarrow \mu^{*}$.

\UNTIL{$|f(\mu^{(k+1)})-f(\mu^{(k)})| \leq \varepsilon$.
	
\STATE Compute $\hat{w}^*$ using \eqref{eq:mu}, and return 
 $\hat{W}^{*} \leftarrow \diag(\hat{w}^*)$.}
\end{algorithmic}
	\label{alg}
\end{algorithm}

The initial condition ${\mu}^{(0)}$ can be chosen as an arbitrary vector with all strictly positive entries. With \eqref{eq:mu}, it will guarantee $\hat{W}^{0} \in \mathcal{M}$, i.e., the reduced graph is balanced. Furthermore,
we can also initialize $\mu$ using the outcome of graph clustering projection in \cite{XiaodongCDC2017Digraph,Xiaodong2020Digraph}. Specifically, from a given clustering $\Pi$, we construct an initial reduced Laplacian matrix in \eqref{sysrMs} as $\hat{L}_b^{(0)}: = \Pi^\top L_b \Pi$, with $L_b$ the Laplacian matrix of the balanced graph $\mathcal{G}_b$. Then, the initial weight of the edge $(i,j)$ in the quotient graph $\hat{\mathcal{G}}_b$ is the $(i,j)$-th entry of $\hat{L}_b^{(0)}$. 
By doing so, $\mu^{(0)}$ can be formed such that $\hat{W}^{0} \in \mathcal{M}$. 


The convergence analysis of Algorithm~\ref{alg} follows naturally from \cite{dinh2011combining}, and it means that a local optimum can be obtained. More importantly, if we select the initial condition from the clustering-based projection, it is guaranteed that, through iteration, the approximation accuracy of reduced-order network model with the weights obtained by Algorithm~\ref{alg} will be improved. In this sense, the approximations obtained by the proposed method is at least better than the ones obtained by clustering-based projection methods in e.g., \cite{XiaodongCDC2017Digraph,Xiaodong2020Digraph}. We will show this merit from a numerical example in the next section.


\section{Illustrative Example} 
\label{sec:example}

To illustrate the effectiveness of the proposed edge weighting approach, we implement it to a sensor network example from \cite{Scutari2008SensorNetworks,XiaodongCDC2017Digraph}. The topology of the this network is shown in Fig.~\ref{originalnetwork}, which consists of 14 strongly connected vertices, and all the edge weights are 1. In this example, two external input signals are injected into the network via vertices 2 and 7, respectively, and the states of vertices 9 and 10 are measured.

Suppose that 5 clusters are given for this directed network as
$
		\mathcal{C}_1 = \{1, 3, 4, 5\}, \
		\mathcal{C}_2  = \{2\}, \ \mathcal{C}_3 = \{6, 8, 9\}, \
		\mathcal{C}_4 = \{7\}, \ \text{and} \ \mathcal{C}_5=\{10,11,12,13,14\},
$
which leads to the quotient network in Fig.~\ref{reducednetwork}, with incidence matrix
\begin{equation*}
	\hat{\mathcal{B}}: = 
\left[\scriptsize	{\begin{array}{cccccccc}
	1 & 1 & -1 &-1 & 0 &  0 &  0 & -1 \\
	-1& 0 &  1&  0&  0&   0 & 0 &  0\\
	0& -1 & 0 & 1 & 1 &  1 & -1 &  0 \\
	0&  0 & 0 & 0 &-1 &  0 &  1 &  0 \\
	0&  0 & 0 & 0 & 0 &  -1 &  0 & 1 \\
	\end{array}}\right],
\end{equation*}
There are 8 edges in the quotient graph, and each edge is assigned with a symbolic weight as labeled in Fig.~\ref{reducednetwork}. These variables, determining the reduction error, are to be determined by our optimization approach.

First, the parameterized reduced model in \eqref{sysbalr} of the quotient graph is generated with matrices
\begin{align*}
	&	\hat{L}_b  = 
	\left[\scriptsize	{\begin{array}{ccccc}
	\hat{w}_1 + \hat{w}_2 &  -\hat{w}_1 & -\hat{w}_2  & 0 & 0\\
	-\hat{w}_3 &  \hat{w}_3   & 0 & 0 & 0 \\
	- \hat{w}_4 & 0 & \hat{w}_4 + \hat{w}_5 + \hat{w}_6 & -\hat{w}_5 & -\hat{w}_6 \\
	0 & 0& -\hat{w}_7 &  \hat{w}_7 &  0\\
	-\hat{w}_8 & 0 & 0 &  0 & \hat{w}_8 \\
	\end{array}}\right], \\
	&\hat{M} = \left[\scriptsize {\begin{array}{cccccc}
	2.1921&
	0.6803&
	0.3779&
	0.0756&
	0.4157
	\end{array}}\right], \\
	& \hat{F}_b =  
	\left[\scriptsize {\begin{array}{ccccc}
	0 & 0.6803 & 0 & 0 & 0\\
	0 & 0 & 0 & 0.0756 & 0 \\
	\end{array}}\right]^\top, 
 	\hat{H}  = \left[\scriptsize {\begin{array}{ccccc}
	0 & 0 & 1 & 0 & 0\\
	0 & 0 & 0 & 0 & 1 \\
	\end{array}}\right],
\end{align*}
and the weight vector $\hat{w}$ in \eqref{omeg} satisfy the following constraints for a balanced graph:
$
\hat{w}_1 = \hat{w}_3, \ 
\hat{w}_2 = \hat{w}_4 + \hat{w}_8, \ 
\hat{w}_6   = \hat{w}_8, \ 
\hat{w}_5  = \hat{w}_7. 
$

Next, we implement Algorithm~\ref{alg} to solve the optimization problem \eqref{pro:min-1} with $\mu = [\hat{w}_1, \hat{w}_2, \hat{w}_5, \hat{w}_6]^\top \in \mathbb{R}_+^4$ as the optimization variable. Particularly, the SeDuMi solver \cite{sturm1999SeDuMi} is adopted to solve the convex problem \eqref{pro:lin-min1}. We choose the initial edge weights obtained by the clustering-based projection \cite{XiaodongCDC2017Digraph,Xiaodong2020Digraph}, which gives
$\hat{w}^{(0)} = [0.6803, 
0.2268,
0.6803,
0.0756,
0.0756,
0.1512,
0.0756,
0.1512]$ 
and the approximation error $\lVert G_e(s)\rVert_{\mathcal{H}_2} = 0.0322$. With $\hat{\delta} = \varepsilon = 10^{-5}$, Algorithm~\ref{alg} stops after 72 iterations. The convergence trajectory of the resulting $\mathcal{H}_2$ reduction error is shown in Fig.~\ref{err}. The final solution of the edge weights are given as
$
\hat{w}^* = \left[  
    0.6826,
	0.2394,
	0.6826,
	0.0948,
	0.0537,
	0.1446,
	0.0537,
	0.1446
 \right]
$, which provides the approximation error $\lVert G_e(s)\rVert_{\mathcal{H}_2} = 0.0187$. 
Through iteration, the edge weighting method further reduces the error by 41.93\%, compared to the clustering-based projection. Therefore, our method can provide a reduced network systems with a better $\mathcal{H}_2$ approximation error.

\begin{figure}[!tp]\centering
	\includegraphics[scale=.6]{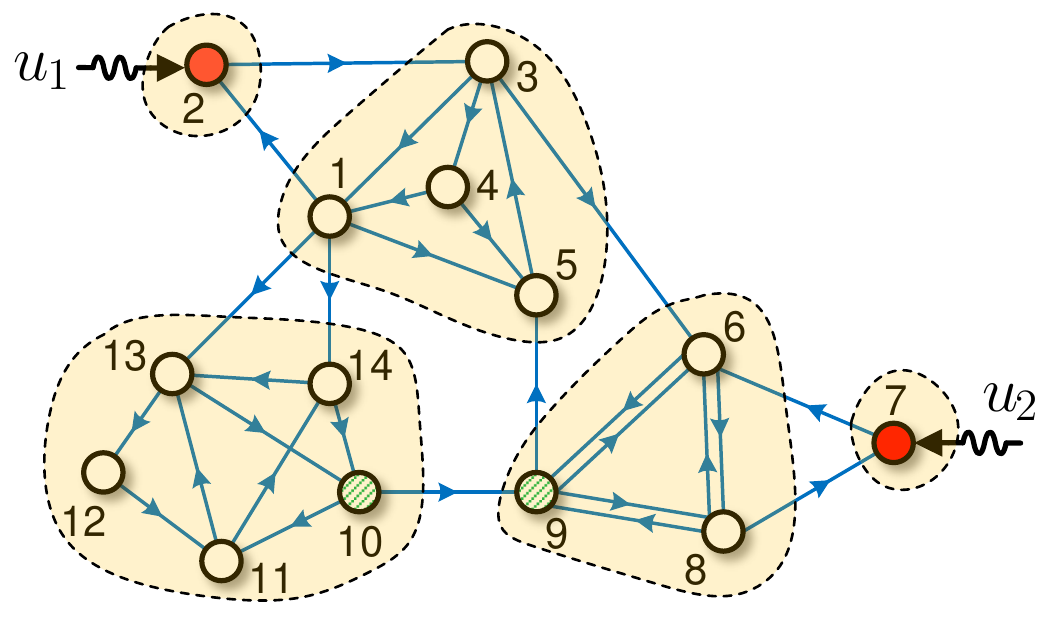}	
	\caption{A connected directed sensor network containing 14 vertices, in which the red vertices are controlled, and the shadowed ones are measured.}
	\label{originalnetwork}
\end{figure}

\begin{figure}[!tp]\centering
	\includegraphics[scale=.7]{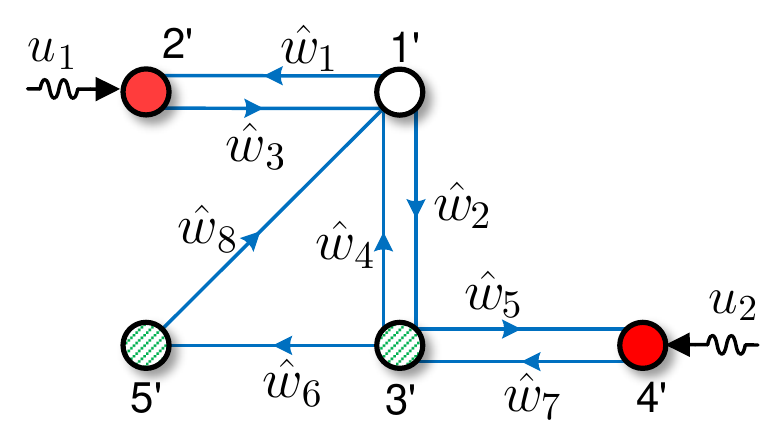}	
	\caption{The quotient graph obtained by clustering, where the controlled vertices are labeled with red color, and measured vertices are indicated by shadow. The weights of the edges are parameters to be determined.}
	\label{reducednetwork}
\end{figure}

\begin{figure}[!tp]\centering
	\includegraphics[scale=.26]{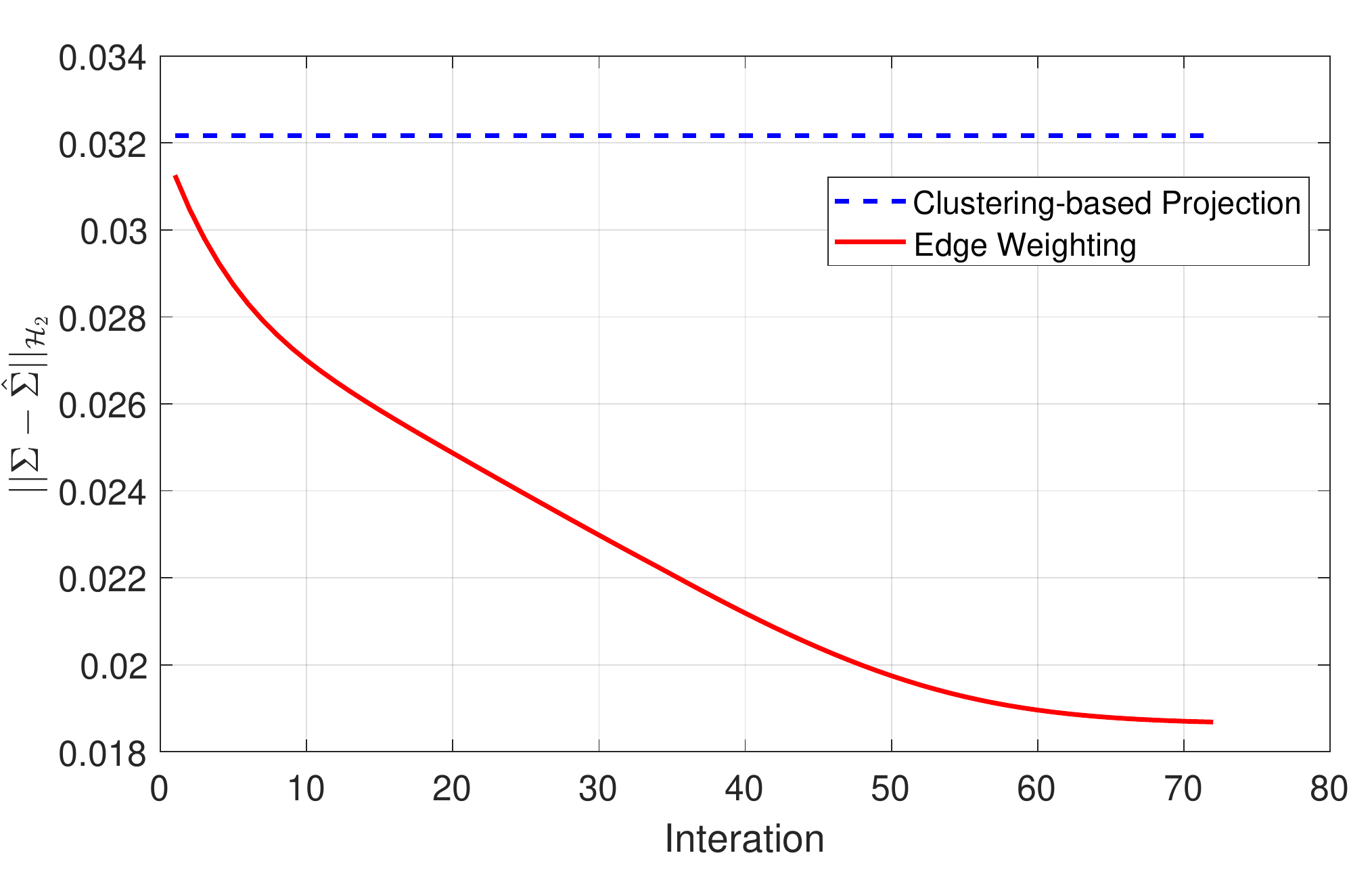}	
	\caption{Approximation errors of clustering-based projection and
		the proposed edge weighting method.}
	\label{err}
\end{figure}

\section{Conclusions} 
\label{sec:conclusion}

In this paper, the $\mathcal{H}_2$ model reduction problem for dynamical networks consisting of diffusively coupled agents has been  formulated as a minimization problem, in which the edge weights in the reduced network are parameters to be chosen. Necessary and sufficient conditions have been proposed for constructing a set of optimal edge weights. An iterative algorithm has been provided to search for the desired edge weights such that the $\mathcal{H}_2$ norm of the approximation error is small. Finally, compared with the projection-based method in \cite{Monshizadeh2014}, the feasibility of this method is illustrated by an example. The advantage of this proposed model reduction method is that not only the structure of the original network has been preserved but also the approximation error has been optimized. 
For future works, we will improve the effectiveness of the iterative algorithm such that the obtained solution is not restricted to a local optimum. Moreover, an extension to networked high-order linear subsystems are also of interest.

%
%



\bibliographystyle{IEEEtran}
\bibliography{weight}

\end{document}